\documentclass[12pt,twoside]{article}
\title{Relative coherent modules}
\date{}
\author{}
\usepackage{graphicx,fancyhdr,fancybox,rotating,xcolor,appendix}
\usepackage{ulem,hhline,dsfont,pifont,multicol,lmodern}
\usepackage{amsthm,amsbsy,geometry,times,pst-node}
\usepackage{amsmath,tikz-cd}

\usepackage{lipsum}
\usepackage{euscript}
\usepackage{textcomp}
\usepackage{pdfpages}
\usepackage{stmaryrd}
\usepackage[all]{xy}
\usepackage{amsfonts}
\usepackage{amsmath}
\usepackage{amssymb}
\usepackage{latexsym}
\usepackage{mathrsfs}

\newtheorem{thm}{Theorem}[section]
 \newtheorem{cor}[thm]{Corollary}
 \newtheorem{lem}[thm]{Lemma}
 \newtheorem{prop}[thm]{Proposition}
 \newtheorem{Def}[thm]{Definition}

 \newtheorem{exs}[thm]{Examples}

\newcommand{\X}{\rm \mathscr{X}}
\newcommand{\Y}{\rm \mathscr{Y}}

\def\Tor{{\rm Tor}}
\def\Hom{{\rm Hom}}
\def\Im{{\rm Im}}
\def\Ker{{\rm Ker}}
\def\Coker{{\rm Coker}}

\begin{document}

\thispagestyle{empty}

\maketitle \vspace*{-1.5cm}
\begin{center}{\large\bf  Mostafa Amini$^{1}$, Arij Benkhadra$^{2,a}$, Driss   Bennis$^{2b}$ and Mohammed El
Hajoui$^{3}$}
\bigskip

\small{1.  Department of Mathematics, Payame Noor University, Tehran, Iran.\\
E-mail: amini.pnu1356@gmail.com}\\

 \small{2. Faculty of Sciences,
 Mohammed V  University in Rabat,   Rabat, Morocco.\\
$\mathbf{a.}$ benkhadra.arij@gmail.com\\
$\mathbf{b.}$ driss.bennis@um5.ac.ma; driss$\_$bennis@hotmail.com
}\\
\small{3. Facult\'{e} Polydisciplinaire de Taroudant, Universit\'{e}  Ibn Zohr, Agadir, Maroc.\\
 hajoui4@yahoo.fr}
\end{center}

\bigskip

\noindent{\large\bf Abstract.}
 Several authors have introduced various type of coherent-like rings and proved analogous results on these rings. It appears that all these relative
 coherent
 rings and all the used techniques can be unified. In \cite{D.Benn}, several coherent-like rings are unified. In this manuscript we continue this work
 and we
 introduce coherent-like module which  also emphasizes our point of view by unifying the existed relative coherent concepts. Several classical results
 are
 generalized and some new results are given.\bigskip

\medskip

\bigskip\bigskip
%

\section{Introduction}
\ \ \
Throughout this paper $R$ will be an associative (non necessarily commutative) ring with identity, and all modules will be unital left $R$-modules (unless specified otherwise).
 In this section, first some fundamental concepts and notations are stated.
 Let  $n$ be a non-negative integer and $M$ an $R$-module. Then
$M$ is said to be $n$-{\it presented}   if there is an exact sequence of   $R$-modules
 $ F_{n}\rightarrow F_{n-1}\rightarrow\dots\rightarrow F_1\rightarrow F_0\rightarrow
M\rightarrow 0$ , where each $F_i$ is a finitely generated free module. In particular, $0$-presented and $1$-presented modules are finitely generated and finitely presented modules, respectively.
$M$ is said to be infinitely presented   if it is $n$-presented  for every positive integer $n$.
A ring $R$ is called (left) {\it coherent}, if every finitely generated (left)  ideal is finitely presented, equivalently every finitely presented $R$-module is $2$-presented and so infinitely presented. The coherent rings were first appeared in Chase's paper \cite{Chase} without being mentioned by name. The term coherent was first used by Bourbaki in \cite{Bou}. Since then, coherent rings have became a vigorously active area of research. For background on coherence for commutative rings, we refer the reader to \cite{Glaz}.
A ring $R$ is called (left)  {\it $n$-coherent} ring if every  $(n-1)$-presented (left)   ideal is $n$-presented, equivalently every  $n$-presented  $R$-module
is $(n+1)$-presented. Also, it is clear that $0$-coherent (resp, $1$-coherent) rings are just Noetherian (resp; coherent) rings. The $n$-coherent rings by Costa in \cite{Costa} introduced, for more details see \cite{D.Benn, DKM, W.O, NG, wu}.
In \cite{DKM}, Kabbaj et al. introduced the concept of $n$-coherent modules, and  $M$ is called  $n$-coherent module  if it is $(n-1)$-presented and every $(n-1)$-presented submodule of $M$ is $n$-presented, the $1$-coherent modules are just the coherent modules, see \cite{Bou}.\\
\ \ \ In this paper, we introduce the {\it $n$-$\X$- coherent} modules.   Let $n$ be an integer, $M$  an $R$-module and $\X$ a class of submodules of $M$.
	Then,  $M$ is said to be {\it $n$-$\X$-coherent}  if $\X_{n-1}$ is non empty and every submodule of $\X_{n-1}$ is in $\X_{n}$, where $\X_{n-1}$ and $\X_{n}$ are two classes of $(n-1)$-presented modules and $n$-presented modules in $\X$, respectively. In particular, if $\X$ is a class of $R$-modules and $M=R$, then $R$ is said to be an {\it $n$-$\X$-coherent} ring  if  every $R$-module of $\X_{n}$ is in $\X_{n+1}$ (see \cite{D.Benn}).
Our main aim is to show that the well-known Glaz, Smaili, Dobbes, Mahdou,  Kabbaj, Chase, Greenberg and Scrivanti characterization of coherent modules and coherent rings hold true for any $n$-$\X$-coherent module and any $n$-$\X$-coherent ring. So, in Section 2, first we study some results of  $n$-$\X$-coherent modules on short exact sequences, factor modules, homomorphism of $R$-modules and direct sum of $R$-modules. Also in this section, several results on transfer of $n$-$\X$-coherence are developed and then in end, another characterizations of $n$-$\X$-coherence using the notion of thickness are given (see Theorems \ref{thm-coh}, \ref{thm-product-coherence}, \ref{thm-transfn}, \ref{tensor},  \ref{faith-flat}, \ref{ideals-charac}, \ref{harac} and Proposition \ref{thickness}). Finally, in Section 3, with considering pullback diagram, some characterizations of $n$-$\X$-coherent rings are studied (see Theorems \ref{Thm-Pullback-Coherence} and \ref{Thm-Pullback-Coherence2}).

\section{$n$-$\X$-coherent modules}
\ \ \  Among the many generalizations of the notion of a coherent ring, we recall the following one: $R$ is said to be (left) $J$-coherent, if every finitely generated  (left) ideal of $R$ contained in $Rad(R)$, the radical of $R$, is finitely presented \cite{DLM}. Also, $R$ is said to be  (left) $Nil_{*}$-coherent, if every finitely generated  (left) ideal of $R$ contained in $Nil(R)$, the nilradical of $R$, is finitely presented \cite{NilXO}. Here, we introduce the following definition of coherence which generalizes all the definitions above.

\begin{Def}
    Let $n$ be an integer, $M$  an $R$-module and $\X$ a class of submodules of $M$.  Let $\X_{n-1}$ and $\X_{n}$ be two classes of $(n-1)$-presented modules and $n$-presented modules in $\X$, respectively.
	We say that $M$ is  (left) $n$-$\X$-coherent, if $\X_{n-1}$ is non empty and every  module of $\X_{n-1}$ is in $\X_{n}$.
\end{Def}
\begin{exs}\label{exs}
\begin{enumerate}
{
\item [\rm (1)]If $\X$ is the class of all submodules of $M$ and $n=1$, then $M$ is $n$-$\X$-coherent if and only if it is pseudo coherent. If, in addition, $M$ is finitely generated then $M$ is $n$-$\X$-coherent if and only if it is coherent (see \cite{Harris}).
 \item [\rm (2)]If $\X$ is the class of all submodules of $M$ contained in $Nil(R)M$ and $n=1$, then $M$ is $n$-$\X$-coherent if and only if it is Nil$_{*}$-coherent.
 \item [\rm (3)] Let $R$ be a semisimple ring and let $\X$ be any non empty class of  submodules of an $R$-module $M$. Then, $M$ is $n$-$\X$-coherent for every integer $n$.
 \item [\rm (4)] Let $M$ be an $R$-module  and let $\X$ be a class of all finitely generated projective submodules of $M$. Then, $M$ is $n$-$\X$-coherent.
 \item [\rm (5)]
Let $K$ be a field and $E$ be a $k$-vector space with infinite rank. Consider $R=K\propto E$ the trivial extension of $K$ by $E$. If $\X$ is the class of all $2$-presented $R$-submodules of $M$, then $M$ is $n$-$\X$-coherent,  since every $2$-presented $R$-submodule of $M$ is projective. But, if $\X^i$ is the class of all $1$-presented $R$-submodules of any desirable $R$-module $M_i$, then there is an $R$-module $M_i$ such that  $M_i$  is not $2$-$\X^i$-coherent, since if any $M_i$ is  $2$-$\X^i$-coherent, then $R$ is regular, a contradiction (see \cite{N.I}).
 \item [\rm (6)]
Let $R_{n+1}=R_n\propto M_n$ be the  trivial extension, where $R_i$ is a non-noetherian commutative ring for any $i\geq 0$. Consider $M_0=\frac{R_0}{I}$ for a finitely generated ideal $I$ of $R_0$. If $\X$ is the class of all  $R$-submodules of $M_{n+1}=\frac{R_{n+1}}{M_n}$, then $M_{n+1}$ is not $(n+2)$-$\X$-coheren for every $n\geq 0$ (see \cite{wu}). }
\end{enumerate}
\end{exs}
For a morphism $\varphi:A\rightarrow B$ and a class $\X$ of submodules of $A$, we denote by $\varphi(\X)$ the class of submodules of $B$ of the form $\varphi(N)$ with $N$ in $\X$.\\
The following theorem is a generalization of  \cite[Theorem 2.2.1]{Glaz}.

\begin{thm}\label{thm-coh}
Let $0\longrightarrow M_1 \stackrel{h}{\longrightarrow} M_2 \stackrel{s}{\longrightarrow} M_3\longrightarrow 0$
be an exact sequence of $R$-modules and
  $\X$ and $\Y$ two classes of submodules of $M_1$ and $M_2$, respectively. Then,
\begin{enumerate}
\item [\rm (1)]   $M_{2}$ is $n$-$\Y$-coherent, if $M_{3}$ is $n$-$s(\Y)$-coherent and $M_{1}$ is $n$-coherent.
 \item [\rm (2)] $M_{1}$ is $n$-$\X$-coherent if $M_{2}$ is $n$-$h(\X)$-coherent.
\item [\rm (3)] $M_{3}$ is $n$-$s(\Y)$-coherent if $M_{2}$ is $n$-$\Y$-coherent and $h(M_{1})$ is $(n-1)$-presented in $\Y$.
\end{enumerate}
\end{thm}
\begin{proof}
  {(1) Let $N_{2}$ be an $(n-1)$-presented submodule in $\Y$. Our aim is to prove that $N_{2}$ is $n$-presented.   For that, consider the following commutative diagram with exact rows and columns:
       \begin{displaymath}
\xymatrix{ &0 \ar[d]^{}& 0 \ar[d]^{}&0 \ar[d]^{}& \\
0 \ar[r]& K \ar[r]^{}\ar[d]^{}& N_2 \ar[r]^{}\ar[d]^{}& s(N_2) \ar[r]^{} \ar[d]^{} & 0\\
0\ar[r]& M_1 \ar[r]_{}& M_2 \ar[r]_{}& M_3 \ar[r]_{}& 0 }
\end{displaymath}
        In view of the exactness of the first row, it suffices to show that both $K$ and $s(N_{2})$ are $n$-presented.
        Since $N_{2}$ is an $(n-1)$-presented module in $\Y$, $s(N_2)$ is an $(n-1)$-presented module in $s(\Y)$, so $s(N_2)$ is $n$-presented.
        By the exactness of the top row, $K$ is an $(n-1)$-presented submodule of $M_1$ which is $n$-coherent. Then $K$ is $n$-presented.

   (2)  Let $N_{1}$ be an $(n-1)$-presented submodule in $\X$.
         Since $h$ is injective, $h(N_1)\cong N_1$ which is $(n-1)$-presented, so $h(N_1)$ is an $(n-1)$-presented module of $h(\X)$, and then by hypothesis, $h(N_1)$ is $n$-presented, so is $N_1$.

    (3) Let $N_{3}$ be an $(n-1)$-presented  submodule in $s(\Y)$. Notice that $\Im(h)= \ker(s) \subseteq s^{-1}(N_3)$.  $M_1$ is $n$-presented, since $M_{2}$ is $n$-$\Y$-coherent. So,  we get using the horseshoe lemma to the following diagram is commutative with exact rows and columns:
         \begin{displaymath}
\xymatrix{ &0 \ar[d]^{}& 0 \ar[d]^{}&0 \ar[d]^{}& \\
0 \ar[r]& K_1 \ar[r]^{}\ar[d]^{}& K_2 \ar[r]^{}\ar[d]^{}& K_3 \ar[r]^{} \ar[d]^{} & 0\\
0 \ar[r]& F_{n-1} \ar[r]^{}\ar[d]^{}& F_{n-1} \oplus F'_{n-1} \ar[r]^{}\ar[d]^{}& F'_{n-1} \ar[r]^{} \ar[d]^{} & 0\\
& \vdots \ar[d]^{}& \vdots \ar[d]^{}& \vdots  \ar[d]^{} &  \\
0 \ar[r]& F_0 \ar[r]^{}\ar[d]^{}& F_0 \oplus F'_0 \ar[r]^{}\ar[d]^{}& F'_0 \ar[r]^{} \ar[d]^{} & 0\\
0 \ar[r]& M_1 \ar[r]^{}\ar[d]^{}& s^{-1}(N_3) \ar[r]^{}\ar[d]^{}& N_3 \ar[r]^{} \ar[d]^{} & 0\\
&0 & 0&0 }
\end{displaymath}
Where $F_i$ and $F'_i$ are finitely generated free modules for every $i\in\{0,...,n-1\}$.
        Due to the exactness of the right vertical sequence, it suffices to prove that $K_3$ is finitely generated. For that, it is sufficient to prove that $K_2$ is finitely generated. Since the middle vertical sequence is exact, it suffices to show that $s^{-1}(N_3)$ is $n$-presented.
        We have that $N_3$ is in $s(\Y)$, then $s^{-1}(N_3)$ is in $s^{-1}(s(\Y))$ and so it is in $\Y$. And since $N_3$ is $(n-1)$-presented, so is $s^{-1}(N_3)$ and consequently, it is $n$-presented. This implies that $K_2$ is finitely generated as desired.}
\end{proof}
Consider a short exact sequence as in Theorem \ref{thm-coh}. If for some class $\Y$ of submodules of $M_2$, we have that $s(\Y)=0$, then $M_2$ is $n$-$\Y$-coherent if $M_1$ is $n$-coherent. For example, for $I=ann(M_3)$ and $\Y$ the class of submodules of $IM_2$, it is evident that $M_2$ is $n$-$\Y$-coherent if and only if $IM_2$ is $n$-coherent. This can be seen just by the definition of $n$-coherence and also if we take in Theorem \ref{thm-coh}, the short exact sequence $0\longrightarrow IM_2\stackrel{h}{\longrightarrow} M_2 \stackrel{s}{\longrightarrow} \frac{M_{2}}{IM_{2}}\longrightarrow 0$.\\

In what follows, for a submodule $N$ of an $R$-module $M$ and a class $\X$ of submodules of $M$, we will denote by $\frac{\X}{N}$ the class of quotient modules
$\frac{L}{N},$ where $L\in \X$ and contains $N$.
The following corollary generalizes \cite[Corollary 2.3]{KDM}.
\begin{cor}
 Let $M$ be an $R$-module, $N$ a submodule of $M$ and $\X$ a class of submodules of $M$. Then, the following assertions hold:
 \begin{enumerate}
   \item[\rm (1)] If $M$ is $n$-$\X$-coherent, $N$ is $(n-1)$-presented and each module in $\X$ contains $N$, then $\frac{M}{N}$ is $n$-$\frac{\X}{N}$-coherent.
  \item[\rm (2)] Assume that $\frac{\X}{N}$ is non empty. Then $M$ is $n$-$\X$-coherent if $\frac{M}{N}$ is $n$-$\frac{\X}{N}$-coherent and $N$ is $n$-coherent.
 \end{enumerate}
\end{cor}
\begin{proof}
Let $ \pi: M \rightarrow \frac{M}{N} $ be the canonical surjection.
 It is evident that, if $\X$ is the class of submodules $K$ of $M$ containing $N$,
then $\frac{\X}{N}$ is the class of quotient modules $\frac{K}{N}$ with $K$ is in $\X$.
Applying Theorem \ref{thm-coh} to exact sequence: $0\rightarrow N \rightarrow M \rightarrow \frac{M}{N} \rightarrow 0$, we get the following results.
\end{proof}
\begin{lem}\label{lem}
  Let $\X$ and $\Y$ two classes of submodules of $M$ with $\X\subseteq\Y$. If $M$ is $n$-$\Y$-coherent, then $M$ is $n$-$\X$-coherent.
\end{lem}
For some submodule $K$ of an $R$-module $M$ and a class $\Y$ of submodules of $M$, we denote by $tr_{K}(\Y)$  the class of submodules of $K$ of the form $K \cap Y$ with $Y\in \Y$. Also, we denote by $f(\Y)$ the class of submodules of the form $f(Y)$ with $Y \in \Y$.
The following proposition is a generalization of \cite[Corollary 2.2.2]{Glaz}.
\begin{prop}\label{Cor-Coh-Im-Ker}
  Let $f:M\rightarrow N$ be a homomorphism of $R$-modules and
  $\X$ and $\Y$ two classes of submodules of $M$ and $N$, respectively. Then,
   The following assertions hold:
  \begin{enumerate}
    \item[\rm (1)] If $M$ is $n$-$\X$-coherent, then $ \Ker(f)$ is $n$-$tr_{\Ker(f)}(\X)$-coherent.
    \item[\rm (2)] If $N$ is $n$-$\Y$-coherent, then $ \Im(f)$ is $n$-$tr_{\Im(f)}(\Y)$-coherent.
     \item[\rm (3)] If $M$ is $n$-$\X$-coherent and $\ker(f)$ is an $(n-1)$-presented module in $\X$, then $\Im(f)$ is $n$-$f(\X)$-coherent.
     \item[\rm (4)] If $N$ is $n$-$\Y$-coherent and $\Im(f)$ is an $(n-1)$-presented module in $\Y$,
    then $\Coker(f)$ is $n$-$\frac{\Y}{\Im(f)}$-coherent.
  \end{enumerate}
\end{prop}
\begin{proof}
 The two first assertions follow by applying (2) of Theorem \ref{thm-coh} and Lemma \ref{lem}
   to the following exact sequences:
    $$0\longrightarrow \Ker(f) \longrightarrow M \longrightarrow \frac{M}{\Ker(f)} \longrightarrow 0$$ and
 $$0 \longrightarrow \Im(f) \stackrel{i}{\longrightarrow} N \longrightarrow \Coker(i) \longrightarrow0.$$
The two last assertions follow by applying (3) of Theorem \ref{thm-coh} to the following exact sequences:
  $$0\longrightarrow \Ker(f) \longrightarrow M \longrightarrow \Im(f) \longrightarrow 0$$
   and $$0\longrightarrow \Im(f) \longrightarrow N \longrightarrow \Coker(f) \longrightarrow 0.$$
\end{proof}
Now, we set the result concerning the coherence of the direct sum of modules.
 It generalizes \cite[corollary 2.2.3]{Glaz}.
Let $(M_i)_{i\in I}$ be a family of $R$-modules and $\X^{i}$ a class of
 submodules of $M_i$, for each $i\in I$. We will denote by $\bigoplus_{i \in I} \X^{i}$
the class of modules of the form $\bigoplus_{i \in I} N_i$ with each $N_i$ is in $\X^{i}$.
\begin{thm}\label{thm-product-coherence}
  Let $(M_i)_{i\in\{1,...,m\}}$ be a finite family of $R$-modules and $\X^{i}$ a class of
 submodules of $M_i$ for every $i=1,...,m$. Then, $\bigoplus\limits_{i=1}^{m} {M_i}$ is
 an $n$-$(\bigoplus\limits_{i=1}^{m} \X^{i})$-coherent $R$-module
  if and only if $M_i$ is $n$-$\X_i$-coherent
  for all $i=1,...,m$.
\end{thm}
\begin{proof}
The \textquotedblleft only if\textquotedblright $\ $ part follows easily using Lemma \ref{lem}, (2) of Theorem \ref{thm-coh} and the following exact sequence: $0\rightarrow M_i \rightarrow \bigoplus_{j=1}^{m}{M_j}  \rightarrow \frac{\bigoplus_{j=1}^{m}{M_j}}{M_i} \rightarrow 0$.\\
For the \textquotedblleft if \textquotedblright $\ $ part, consider an $(n-1)$-presented submodule $N$ of $\bigoplus_{i=1}^{m}{M_i}$ and the canonical projection $\pi_{i}: \bigoplus_{i=1}^{m}{M_i} \rightarrow M_i $.
Thus, $\pi_{i}(N)\in\X^{i}_{n-1}$ for all $i=1,...,m$.
Then, $\pi_{i}(N)$ is $n$-presented for all $i=1,...,m$.
 We have the following exact sequence:
$$0\longrightarrow \bigoplus_{j=1,i\neq j}^{m} \pi_{j}(N) \longrightarrow N \longrightarrow \pi_{i}(N) \longrightarrow 0.$$
Consequently by \cite[Theorem 1]{wu}, $N$ is $n$-presented, which completes the proof.
\end{proof}
The following result is a generalization of \cite[Corollary 2.2.5]{Glaz}.
\begin{cor}\label{Hom}
  Let $R$ be a commutative ring, $M$  a finitely generated $R$-module and
  $N$ an $n$-$\X$-coherent module for some class $\X$ of submodules of $N$.
  Then, $\Hom_{R}(M,N)$ is $n$-$\Y$-coherent, where $\Y$ is the class of submodules of $\Hom_{R}(M,N)$ which are isomorphic to a module in $tr_{A}(\sum_{i=1}^{k}\X)$.
\end{cor}
\begin{proof}
  Since $M$ is finitely generated, there is an exact sequence of $R$-modules $\ 0\rightarrow K \rightarrow R^k \rightarrow M \rightarrow 0\ $ for some non-negative integer $k$.
  As $\Hom_{R}(-,N)$ is a left-exact functor, $\Hom_{R}(M,N)\cong A$, where $A$ is a submodule of $\Hom_{R}(R^k,N)\cong N^k$. By Theorem \ref{thm-product-coherence}, $N^k$ is $n$-$\bigoplus_{i=1}^{k}\X$-coherent, and so by Proposition \ref{Cor-Coh-Im-Ker}, $A$ is $n$-$tr_{A}(\sum_{i=1}^{k}\X)$-coherent, which completes the proof.
\end{proof}

We finish this section with some transfer results. First, we present a generalization of \cite[Theorem 2.5]{DKM}.
\begin{thm}\label{thm-transfn}
  Let $I$ be an $(n-1)$-presented two-sided ideal of a ring $R$, $M$  an
  $\frac{R}{I}$-module and $\X$ a class of submodules of $M$. Then, $M$ is
  $n$-$\X$-coherent as an $R$-module if and only if $M$ is $n$-$\X$-coherent
   as an $\frac{R}{I}$-module.
\end{thm}
For the proof, we need the following lemma.
\begin{lem}[\cite{DKM}, lemmas 2.6 and 2.7]\label{lem-trans}
Let $R \rightarrow S$ be a ring homomorphism and $m$ a non-negative integer. Then, the following assertions hold:
\begin{enumerate}
   \item[\rm (1)] If $S$ is $m$-presented as an $R$-module, then every $m$-presented $S$-module is an $m$-presented   $R$-module.
    \item[\rm (2)] If $S$ is $(m-1)$-presented as an $R$-module and $M$ a $S$-module, then $M$ is $m$-presented $S$-module if $M$ is an $m$-presented   $R$-module.
\end{enumerate}
\end{lem}
\textit{Proof of Theorem \ref{thm-transfn}}. $\ $ Let $R \rightarrow \frac{R}{I}$ be the canonical homomorphism and $N$  a submodule of $M$.
Using Lemma \ref{lem-trans}, we get the following equivalences:
$N$ is an $(n-1)$-presented $\frac{R}{I}$-submodule of $\X$ if and only if it is $(n-1)$-presented $R$-submodule of $\X$.
Consequently, $M$ is an $n$-$\X$-coherent as an $\frac{R}{I}$-module if and only if it is $n$-$\X$-coherent as an $R$-module. $\quad \square$\\

Assume that $S \geq R$ is a unitary ring extension. Then,
the ring $S$ is called  right {\it $R$-projective}  in case, for any right $S$-module $M_S$ with
an $S$-module $N_S$, $N_R\mid M_R$ implies $N_S\mid M_S$, where $N\mid M$ means that $N$ is a direct summand
of $M$.
  The ring extension $S \geq R$ is called a {\it finite normalizing extension}  in case there is a finite subset $\{s_1,...,s_n\}\subseteq S$ such that
$S=\sum_{i=1}^{i=n}s_{i}R$ and $s_{i}R=Rs_{i}$ for $i=1,...,n$.
A finite normalizing extension $S \geq R$ is called an {\it almost excellent extension}  in
case $_RS$ is flat, $S_R$ is projective, and the ring $S$ is right $R$-projective (see \cite{wu}).

In the following, we mainly consider the properties of $n$-$\X$-coherent modules and $n$-$\X$-coherent rings under an almost excellent extension of commutative rings.
\begin{thm}\label{thm-transf}
Let  $S\geq R$ be an almost excellent extension, $M$ a $S$-module and $\X$ a class of
 submodules of $M$. Then,  the following statements are equivalent:
\begin{enumerate}
   \item[\rm (1)]
$M$ is $n$-$\X$-coherent as an $R$-module;
 \item[\rm (2)]
$\Hom_{R}(S,M)$ is $n$-$\Y$-coherent, where $\Y$ is the class of submodules of $\Hom_{R}(S,M)$;
 \item[\rm (3)]
$M$ is  $n$-$\X$-coherent as an $S$-module.
\end{enumerate}
\end{thm}
\begin{proof}
{$(1)\Longrightarrow (2)$ $S$ is a finitely generated $R$-module. So (2) follows from Corollary \ref{Hom}.

$(2)\Longrightarrow (3)$ Assume that $N$ is an $(n-1)$-presented submodule of $M$ in $\X$. We show taht $N$ is $n$-presented. By \cite[Lemma 1.1]{du}, $M\cong K$, where $K$ is a direct summand  of ${\rm Hom}_{R}(S,M)$. Therefore by hypothesis and Corollary \ref{Cor-Coh-Im-Ker}, $K$ is $n$-$tr_{K}(\Y)$-coherent. So, we deduce that  $N$ is   $n$-presented.

$(3)\Longrightarrow (1)$ By \cite[Theorem 5]{wu}, $S$ is an $n$-presented $R$-module. So, if $N$ is an $(n-1)$-presented submodule of $R$-module $M$ in $\X$, then by Lemma \ref{lem-trans} and (3), $N$ is an $n$-presented submodule of $R$-module $M$, and hence $M$ is $n$-$\X$-coherent as an $R$-module.
}
  \end{proof}

In what follows, we will denote by $\X_k$, for some class  of $R$-modules $\X$ and an integer $k$, the subclass of $k$-presented submodules of $\X$ (which we assume they exist).
 For an $R$-module $M$ and a class $\X$ of submodules of $M$, we will denote by $^{S\otimes}\X$
the class $\{ S \otimes N$, where $N$ is a module of $\X \}$.
\begin{thm}\label{tensor}
Let  $S\geq R$ be an almost excellent extension, $M$ a $S$-module and $\X$ a class of
 submodules of $M$. Then,  the following statements are equivalent:
\begin{enumerate}
   \item[\rm (1)]
$M$ is $n$-$\X$-coherent as an $R$-module;
 \item[\rm (2)]
$S\otimes_{R}M$ is $n$-$^{S\otimes}\X$-coherent;
 \item[\rm (3)]
$M$ is  $n$-$\X$-coherent as an $S$-module.
\end{enumerate}
\end{thm}
\begin{proof}
{$(1)\Longrightarrow (2)$ Assume that $N$ is an $(n-1)$-presented submodule of $S \otimes_R M$ in $^{S\otimes}\X$. So, there is a  submodule $I$ in $\X$ such that $N=S \otimes_R I$. By \cite [Lemma 4]{wu}, $I\in \X_{n-1}$ as an $R$-module, and so by (1), $I\in \X_{n}$. Hence by Lemma \ref{lem-trans}, we deduce that $N$ is $n$-presented.

$(2)\Longrightarrow (1)$ Assume that $N$ is an $(n-1)$-presented submodule of $M$ in $\X$. Then by \cite [Lemma 4]{wu} and (2), $N\in \X_{n}$.

$(1)\Longrightarrow (3)$ and $(3)\Longrightarrow (1)$ are trivial.
}
  \end{proof}
\begin{cor}\label{extension}
Let  $S\geq R$ be an almost excellent extension and $\X$ a class of
 $R$-modules. Then,  $R$ is $n$-$\X$-coherent if and only if $S$ is $n$-$^{S\otimes}\X$-coherent.
\end{cor}
\begin{proof}
{It is particulary of Theorem \ref{tensor}.}
 \end{proof}

The next result generalizes \cite[Theorem 2.11]{DKM} and \cite[Corollary 2.2.5]{Glaz}.
\begin{thm}\label{faith-flat}
  Let $R \rightarrow S$ be a ring homomorphism making $S$ a faithfully flat
  right $R$-module, $M$  an $R$-module and $\X$ a class of submodules of
  $M$. Then
\begin{enumerate}
 \item[\rm (1)]
 $M$ is $n$-$\X$-coherent,
  if $S \otimes_R M$ is an $n$-$^{S\otimes}\X$-coherent $S$-module.
 \item[\rm (2)]
$M$ is $n$-$\X$-coherent
  if and only if $S \otimes_R M$ is an $n$-$^{S\otimes}\X_{n-1}$-coherent $S$-module.
\end{enumerate}
\end{thm}
\begin{proof}
  (1) Let $N$ be an $(n-1)$-presented module of $\X$, then $S \otimes N$ is an $(n-1)$-presented module of $^{S\otimes}\X$ (since $S$ is flat).
  Then, $S \otimes N$ is $n$-presented, so is  $N$ since $S$ is faithfully flat.

(2)($\Longleftarrow$) This is a direct consequence of (1).

($\Longrightarrow$) Assume that $K$ is an $(n-1)$-presented submodule of $S \otimes_R M$ in $^{S\otimes}\X_{n-1}$. So, there is an $(n-1)$-presented submodule $N$ in $\X_{n-1}$ such that $K=S \otimes_R N$. By hypothesis,   $N$ is in $\X_{n}$, and so by \cite[Theorem 2.1.9]{Glaz}, $K$ is $n$-presented
\end{proof}
\begin{cor}\label{flat}
Let $R\rightarrow S$ be a ring homomorphism
making $S$ a faithfully flat right $R$-module and $\X$ a class of ideals of $R$.
 Then, $R$ is an $n$-$\mathcal{X}$-coherent ring, if  $S$ is an
$n$-$^{S\otimes}\X$-coherent ring.
\end{cor}
\begin{proof}
It is enough to take $M=R$.
\end{proof}
${\mathbf{Question:}}$
 Let $R\rightarrow S$ be a ring homomorphism
making $S$ a faithfully flat right $R$-module and $\X$ a class of ideals of $R$. If $R$ is   $n$-$\mathcal{X}$-coherent, then what conditions on the fibers $R\rightarrow S$  are required in order that $S$ is $n$-$^{S\otimes}\X$-coherent?

Now, we give a generalization of the classical result due to Chase in \cite{Chase} stating that $R$ is coherent if
and only if the annihilator of any element $a$ of $R$ is finitely generated and the intersection of any two finitely generated ideals in $R$ is also finitely generated.\\

We say that a class of modules $\X$ is said to be closed under finite sums
 if, for every finite family of modules $\{Mi\}_{i\in I}$ in $\X$,    $\sum_{i\in I} M_i$ is also in $\X$. A class $\X$ is said   to be closed under cyclic submodules if, whenever  $N$ is a
   cyclic submodules of a module in $\X$, it is also in $\X$.\\

The following theorem generalizes  \cite[Theorem 2.2]{Chase}.

\begin{thm}\label{ideals-charac}
Let $M$ be an $R$-module and let $\X$ be a class of submodules of $M$ such that $\X_n$ is closed under finite sum and closed under cyclic submodules.
Then, $M$ is left $1$-$\X$-coherent if and only if $(0:_R a)$ is a finitely generated  of $M$ for any $a\in M$ such that $Ra$ is in $\X_0$ and the intersection of any two submodules of $M$ in $\X_0$ is finitely generated.
\end{thm}
\begin{proof}
 Suppose that $M$ is $1$-$\X$-coherent and let $a$ be in $M$ such that $Ra$ is in an element $N$ of $\X_0$, then $Ra\in\X_0$. Then, $Ra$ is in $\X_1$. Considering the exact sequence: $0\rightarrow (0:_Ra)\rightarrow R\rightarrow Ra\rightarrow0$, we get that $(0:_Ra)$ is a finitely generated ideal of $R$.\\
  Now, let $N$ and $L$ be in $\X_0$. Then, $N+L\in \X_0$. So by hypothesis,
  $N+L$ is in $\X_1$ and $N\oplus L$ is finitely generated as an $R$-module.
   Via the exact sequence $0\rightarrow N\cap L\rightarrow N\oplus L\rightarrow N+L\rightarrow0$, we get that $N\cap L$ is a finitely generated submodule of $M$.\\
   Conversely, let $N\in\mathcal{X}_0$. Then, there exist $a_1,...,a_p\in M$
   such that $N=\sum_{i=1}^{p}Ra_i$. We prove by induction on $p$ that $N$
   is 1-presented. If $p=1$, $(0:_Ra_1)$ is finitely generated submodule of $M$. Hence, $N$ is 1-presented by the exactness of the sequence
    $0\rightarrow (0:_Ra_1)\rightarrow R\rightarrow N \rightarrow0$. For the induction step (with $p>1$),
     consider the following exact sequence
      $0\rightarrow(\sum_{i=1}^{p-1}Ra_i)\cap Ra_p\rightarrow
       (\sum_{i=1}^{p-1}Ra_i)\oplus Ra_p\rightarrow N\rightarrow0$.
        By hypothesis on $\mathcal{X}_0$, we have $Ra_p$ and
        $\sum_{i=1}^{p-1}Ra_i$ are in $\mathcal{X}_0$,
        then they are in $\mathcal{X}_1$, thus
        $(\sum_{i=1}^{p-1}Ra_i)\oplus Ra_p$ is 1-presented. Therefore, $(\sum_{i=1}^{p-1}Ra_i)\cap Ra_p$ is a finitely generated ideal of $M$, and thus
          $N$ is 1-presented.
\end{proof}
Let $I$ be an ideal of $R$ and $\X$ be the class of ideals $J$ of $R$ contained in $I$. Then, $R$ is $1$-$\X$-coherent if and only if $I$ is quasi-coherent.
\begin{cor}
Let $\X$ be a class of ideals of $R$ such that $\X_n$ is closed under finite sum and closed under cyclic submodules.
Then, $R$ is left $1$-$\X$-coherent if and only if $(0:_R a)$ is a finitely generated  of $R$ for any $a\in R$ such that $Ra$ is in $\X_0$ and the intersection of any two ideals in $\X_0$ is finitely generated.
\end{cor}
\begin{cor}
Let $I$ be an ideal of $R$. Then, $I$ is quasi-coherent if and only if $(0:_{R} a)$ is a finitely generated ideal of $R$ for any $a\in I$ and the intersection of any two left (resp., right) ideals contained in $I$ is finitely generated.
\end{cor}
As an application of the previous results established in this section, we get the following result on the coherence of the amalgamated algebra alon an ideal which is proved differently in \cite{KDM}.
\begin{cor}
Let $R_1$ and $R_2$ be two unitary associative rings and let $f:R_1\rightarrow R_2$ be a ring homomorphism. Let $J$ be a finitely generated ideal of $Nil(R_2)$ such $f^{-1}(J)$ is a finitely generated ideal of $R_1$. Then,
 $R_1\bowtie^{f}J$ is $1$-$Nil$-coherent $R_1$-module if and only if $R_1$ and $f(R_1)+J$ are $1$-$Nil$-coherent.
\end{cor}
\begin{proof}
  The direct implication is proved directly using corollary \ref{Cor-Coh-Im-Ker} and the fact that $p_{R_2}(A\bowtie^{f}J)=f(A)+J$ and $p_{R_1}(A\bowtie^{f}J)=A$ for any  ideal $A$ of $R_1$, where $p_{R_1}$ and $p_{R_2}$ are respectively the projection of $R_1\bowtie^{f}J$ on $R_1$ and $R_2$.\\
  For the inverse, in light of theorem \label{thm-ideals-charac}, it sufficient to prove that $(0:(a,f(a)+j))$ is a finitely generated  of $R_1\bowtie^f J$ for any $a\in R_1$ and $j\in J$ such that $R(a,f(a)+j)$ is in the nil-radical of $R_1\bowtie^f J$ and the intersection of any two submodules of $R_1\bowtie^f J$ in the nilradical is finitely generated. For that, it is easy to prove that $(0:_{R_1}(a,f(a)+j))=(0:_{R_1}a)\cap(0:_{R_1}f(a)+j)$ and for any two ideals of $R_1$, we have $(N\bowtie^f J)\cap(I\bowtie^f J)= (N\cap I) \bowtie^f J$.
\end{proof}
Now, we give some transfer results. We start with a generalization of  \cite[Theorem 2.13]{DKM} and \cite[Theorem 2.4.3]{Glaz}.
Let $(M_i, i\in\{1,...,p\})$ be a family of modules and $\X^i$ a class of submodules of $M_i$ for $i\in\{1,...,p\}$.
We will denote by $\prod\limits_{i=1}^p\X^i$ the class of the submodules $\prod\limits_{i=1}^p N_i$ with each $N_i$ is in $\X^i$.
\begin{thm}\label{harac}
Let $(R_i, {1\leq i\leq p})$ be a family of rings. Let $(M_i, {1\leq i\leq p})$ be a family of $R_i$-modules, $p\geq1$ an integer, $\X^i$ a class of submodules of $M_i$ for any integer $i\in\{1,...,p\}$ and $\X=\prod\limits_{i=1}^p\X^i$.
Then, $M=\prod\limits_{i=1}^p M_i$ is $n$-$\X$-coherent $R$-module if only if $M_i$ is    $n$-$\X^i$-coherent  $R_i$-module
for each $i=1,...,p$, where $R=\prod\limits_{i=1}^p R_i$.
\end{thm}
\begin{proof}
($\Longrightarrow$) Let $p=2$. Consider the  short exact sequence
$ 0\rightarrow R_2 \rightarrow R_{1}\times R_{2} \rightarrow R_1 \rightarrow 0, $  where $R_1$ is an $n$-presented $R_{1}\times R_{2}$-module, since by \cite[Lemma 2.14]{DKM}, $R_1$ is an infinitely presented $R_{1}\times R_{2}$-module . So, if $N_1$  is a submodule of $M_1$ in $\X^{1}_{n-1}$, then by Lemma \ref{lem-trans}(1), $N_1$ is in $\X_{n-1}$, and so $N_1$ is in $\X_{n}$. Therefore by Lemma \ref{lem-trans}(2), $N_1$ is in $\X^{1}_{n}$. Similary, if $N_2$  is a submodule of $M_2$ in $\X^{2}_{n-1}$, then $N_2$ is  in $\X^{2}_{n}$.

($\Longleftarrow$)
Suppose that, for every $i \in \{1,...,p\}$, $M_i$ is $n$-$\X^i$-coherent. Let $N$ be a module of $\X_{n-1}$. Then there exist
$N_1,...,N_p$ in $\prod\limits_{i=1}^p\X^i$ such that
$N=\prod\limits_{i=1}^p N_i$. For each $i\in\{1,...,p\}$, $N_i$ is in $\X_{n-1}$. And so, $N_i$ is in $\X^{i}_{n-1}$. Hence, it
is in $\X^{i}_{n}$ by $n$-$\X_i$-coherence of $M_i$. Consequently, $N$ is also in  $\X_n$, and so we deduce that $M$ is an $n$-$\X$-coherent $R$-module.
\end{proof}
\begin{exs}
{Let $(R_i, {1\leq i\leq n})$ be a family of rings in Example \ref{exs}(6), and also let $(M_i, {1\leq i\leq n})$ be a family of $R_i$-modules in Example \ref{exs}(6). Consider $R=\prod\limits_{i=1}^n R_i$ and $M=\prod\limits_{i=1}^n M_i$. If $\X^i$ is a class of submodules of $M_i$ for any integer $i\in\{1,...,n\}$ and $\X=\prod\limits_{i=1}^n\X^i$, then by Example \ref{exs}(6) and Theorem \ref{harac}, $M$ is not $(n+1)$-$\X$-coherent $R$-module.}
\end{exs}
\begin{cor}
Let $(R_i, {1\leq i\leq p})$ be a family of rings, $p\geq1$ an integer,. Let $\X^i$ be a class of    ideals of $R_i$ for any integer $i\in\{1,...,p\}$ and
$\X=\prod\limits_{i=1}^p\X^i$.
Then, $R=\prod\limits_{i=1}^p R_i$ is   $n$-$\X$-coherent if only if $R_i$ is    $n$-$\X^i$-coherent
for each $i=1,...,p$.
\end{cor}

We end this section by establishing another characterization of $n$-$\X$-coherence using the notion of thickness. A class of modules $\Y$ is said to be thick if it is closed under direct summands and whenever we are given a short exact sequence $0 \rightarrow A \rightarrow B \rightarrow C \rightarrow 0$
with two out of the three terms $A, B, C$ in $\Y$, so is the third module. In \cite[Theorem 2.5.1]{Glaz} and \cite[Theorem 2.4]{BPCot}, it is proved that when $R$ is coherent, the class of $n$-presented $R$-modules is thick. Here, we set the following generalization.
\begin{prop}\label{thickness}
Let $n$ be a non negative integer and $\X$ a class of $R$-modules which is closed under direct summand and kernels of epimorphisms. Then, the following assertions are equivalent:
 \begin{enumerate}
   \item [\rm (1)] $R$ is   $n$-$\X$-coherent;
   \item [\rm (2)] The class $\X_{n-1}$ is thick;
   \item [\rm (3)] $\X_{n-1}=\X_{\infty}$.
  \end{enumerate}
\end{prop}
\begin{proof}
  $(3)\Longrightarrow(2)$. It suffices to show that $\X_{\infty}$ is thick which is easily deduced using \cite[Theorem 2.1.2]{Glaz}, since $\X_{\infty}=\bigcap_{k\geq0}\X_k$.\\
  $(2)\Longrightarrow(1)$. Let $I\in \X_{n-1}$, then there is an exact sequence of $R$-modules
  $0\rightarrow K\rightarrow F_0\rightarrow I\rightarrow 0$, where $K\in \X_{n-2}$ and $F_0$ is finitely generated and free.
  Since $\X_{n-1}$ is thick and both $I$ and $F_0$ are in $\X_{n-1}$, we get that $K \in \X_{n-1}$ and so $I \in \X_n$.\\
  $(1)\Longrightarrow(3)$. Let $I\in \X_{n-1}$. By the coherence of $R$, $I\in \X_n$. Using the same argument as in $(2)\Longrightarrow(1)$, we can obtain
  an $(n+1)$-presentation of $I$. Iterating this procedure yields a finite $m$-presentation of $I$ for all $m\geq n$. Hence $I\in\bigcap_{m\geq0}\X_m=\X_{\infty}$.
\end{proof}

\section{On the coherence of pullbacks}
By a ring, we mean a commutative ring with identity. Considering a commutative square of rings and ring homomorphisms of the following form :
\begin{equation}\label{pull-diag}
\begin{tikzcd}
R\arrow[r,"i_1"] \arrow[d,"i_2"] & R_1 \arrow[d,"j_1"] \\
R_2 \arrow[r,"j_2"] & R_0
\end{tikzcd}
\end{equation}
Recall that \eqref{pull-diag} is called a pullback diagram if $R=\ker(j_1\circ p_1 - j_2 \circ p_2)$, where $p_1$ be the projection of $R$ on $R_1$ and $p_2$ be the projection of $R$ on $R_2$.

In the following, we say that a class $\X$ of modules satisfies the property $(*)$ proper if for every module $M\in \X_1$, there exists an exact sequence $0\rightarrow K\rightarrow R^k \rightarrow M \rightarrow 0$ with $K\in\X$. We, also, consider a pullback diagram \eqref{pull-diag} with $i_1$ is surjective. \\
The following lemma can be found in \cite[Proposition 1(c)]{Susana}.
\begin{lem}\label{finitness}
An $R$-module $M$ is finitely generated if and only if $R_i\otimes_R M$ is a finitely generated $R_i$-module for $i=1,2$.
\end{lem}
The following proposition generalizes \cite[Proposition 2.1]{Greenberg} and \cite[Proposition 2]{Susana}.
\begin{prop}\label{prop-finiteness-pullback}
Let $M$ be an $R$-module and let $\X$ be a class of $R$-modules satisfying the property $(*)$ and let $\Y^i$ be a class of $R_i$-modules  such that $^{R_i\otimes}\X_k$ is a subclass of $\Y^{i}_k$ for every integer $k$ and $i=1,2$.\\
Suppose that $\Tor_j^R(R_i,M)$ is in $\Y_{n-j}^i$ for $1\leq j\leq n$, then $M$ is in $\X_n$ if and only if $R_i\otimes_R M$ is in $\Y_n^i$ for $i=1,2$.
\end{prop}
\begin{proof}
 We use the induction on $n$.
  The case $n=0$ follows easily from Lemma \ref{finitness}.\\
  Now, assume that $\Tor_j^R(R_i,M)$ is in $\Y_{n-j}^i$ for $1\leq j\leq n$ and $i=1,2$.\\
  Let $M$ is in $\X_n$. Then, we have an exact sequence of $R$-modules
 $$0\longrightarrow K \stackrel{i}{\longrightarrow}R^{k} \stackrel{f}{\longrightarrow} M\longrightarrow 0.   \ \ \ \ \ \ \  (a)$$
  We have that $\Tor_j^R(R_i,K)\cong\Tor^R_{j+1}(R_i,M)$ for any $j\geq1$. So, $R_i\otimes K$ is in $\Y_{n-1}^i$.\\
  Now, we tensor the short exact sequence (a) with $R_i$ over $R$ and we obtain the following two exact sequences:
  $$0\longrightarrow K_i\longrightarrow R_i^k\longrightarrow R_i\otimes_R M\longrightarrow0 $$
and
 $$0\longrightarrow\Tor_1^R(R_i,M)\longrightarrow R_i\otimes_R K\longrightarrow R_i\otimes_R K_i\longrightarrow0 $$
with $K_i=\ker(\mathds{1}_{R_i}\otimes f)$.\\
From the previous two exact sequences, we can deduce that $K_i$ is in $\Y_{n-1}^i$ which implies that $R_i\otimes_R M$ is in $\Y_{n}^i$ for $i=1,2$.\\
For the converse, suppose  that $R_i\otimes_R M$ is in $\Y_{n}^i$, then $K_i$ is in $\Y_{n-1}^i$. And so $R_i\otimes K$ is in $\Y_{n-1}^i$. Hence $M$ is in $\X_n$.
\end{proof}
The following theorem generalizes \cite[Theorem 4]{Susana}.
\begin{thm}\label{Thm-Pullback-Coherence}
Let $M$ be an $R$-module and let $\X$ be a class of $R$-modules satisfying the property $(*)$ and let $\Y^i$ be a class of $R_i$-modules  such that $^{R_i\otimes}\X_k$ is a subclass of $\Y^{i}_k$ for every integer $k$ and $i=1,2$.\\
Suppose that for each $M\in \X_n$, $\Tor_j^R(R_i,M)$ is in $\Y_{n+1-j}^i$ for $1\leq j\leq n+1$ and $i=1,2$.
Then, $R$ is $n$-$\X$-coherent if $R_i$ is $n$-$\Y^i$-coherent.
\end{thm}
\begin{proof}
  Let $M$ be an $R$-module in $\X_n$. By Proposition \ref{prop-finiteness-pullback}, $R_i\otimes_R M$ is in $\Y_n^i$. Then, by the coherence of $R_i$, $R_i\otimes_R M$ is in $\Y_{n+1}^i$. Again, by Proposition \ref{prop-finiteness-pullback}, $M$ is in $\X_{n+1}$, and hence $R$ is $n$-$\X$-coherent.
\end{proof}
\begin{cor}\label{Cor-Pullback-Coh-Ideals}
Let $M$ be an $R$-module and let $\X$ be a class of $R$-modules satisfying the property $(*)$ and let $\Y^i$ be a class of $R_i$-modules such that $^{R_i\otimes}\X_k$ is a subclass of $\Y^{i}_k$ for every integer $k$ and $i=1,2$.\\ Suppose that $R_i$ is $n$-$\Y^i$-coherent.
Then, $R$ is $n$-$\X$-coherent if and only if for all $I\in \X_n$,  $\Tor_j^R(R_i,\frac{R}{I})$ is in $\Y_{n+1-j}^i$ for $1\leq j\leq n+1$ and $i=1,2$.
\end{cor}
\begin{proof}
    The only if assertion follows from Theorem \ref{Thm-Pullback-Coherence}, we will prove the converse.\\
  Let $I\in\X_n$, then we have an exact sequence of the form
 $$0\longrightarrow I \longrightarrow R \stackrel{\pi}{\longrightarrow} \frac{R}{I}\longrightarrow 0 \ \ \ \ \ \ \ \ \ \ \  \ \ \ \ \  \ \ \ \ \ \  \ \  (1)$$
  Tensoring the sequence (1) with $R_i$ $(i=1,2)$ over $R$ and put $H_i=\ker(\mathds{1}_{R_i}\otimes\pi)$, we obtain two exact sequences
$$0\longrightarrow H_i \longrightarrow R_i\stackrel{\mathds{1}_{R_{i}}\otimes\pi}{\longrightarrow} R_i\otimes_R \frac{R}{I} \longrightarrow 0 \ \ \ \ \ \  \ \ \ \ \ \ \ \ \  \ \ \ (2)$$
and
$$0\longrightarrow {\rm Tor}_1^R(R_i,\frac{R}{I}) \longrightarrow R_i\otimes_R I \longrightarrow H_i \longrightarrow 0. \ \ \ \ \ \ \ \ \ \ (3)$$
Using the coherence of $R$ and the exactness of the sequences (1), (2) and (3), we get that $\Tor_1^R(R_i,\frac{R}{I})$ is in $\Y^i_n$.\\
Now, since $I$ is in $\X_n$, we have an exact sequence
$$0\longrightarrow P\longrightarrow R^s \longrightarrow I \longrightarrow 0.$$
Using a similar argument, we get that $\Tor_1^R(R_i,I)$ is in $\Y^i_{n-1}$, and hence $\Tor_2^R(R_i,\frac{R}{I})$ is also in $\Y^i_{n-1}$.
Since $P$ is in $\X_{n-1}$, we have an exact sequence
$$0\longrightarrow P_0 \longrightarrow R^r \longrightarrow P \longrightarrow 0.$$
Using a similar argument, we get that $\Tor_1^R(R_i,P)$ is in $\Y^i_{n-2}$, and hence $\Tor_2^R(R_i,I)$ is also in $\Y^i_{n-2}$.
Iterating with the same argument, we get that for each $I\in \X_n$,  $\Tor_j^R(R_i,\frac{R}{I})$ is in $\Y_{n+1-j}^i$ for $1\leq j\leq n+1$ and $i=1,2$.
\end{proof}
Now, we establish necessary and sufficient conditions for the coherence of the pullback diagram.
\begin{thm}\label{Thm-Pullback-Coherence2}
Let  $\X$ be a class of $R$-modules satisfying the property $(*)$ and let $\Y^i$ be a class of $R_i$-modules for $i=1,2$.\\
Suppose that for each $M\in \X_n$,  $\Tor_j^R(R_i,M)$ is in $\Y_{n+1-j}^i$ for $1\leq j\leq n+1$ and $i=1,2$. Also, let for any module in $Y_i\in\Y^i_n$, there exists a module $X_i\in\X_n$ such that $R_i\otimes_R X_i\simeq Y_i$.
Then, $R$ is $n$-$\X$-coherent if and only if  $R_i$ is $n$-$\Y^i$-coherent.
\end{thm}
\begin{proof}
 The direct sense of the equivalence is proved in Theorem \ref{Thm-Pullback-Coherence}.
  For the converse, let $N_i$ be a $R_i$-module of $\Y^{i}_n$, $(i=1,2)$. By hypothesis, $N_i=R_i\otimes_R N^{'}_i$, where $N^{'}_i$ is in $\X_n$. Then, by the coherence of $R$, $N^{'}_i$ is in $\X_{n+1}$. Hence, $N_i$ is in $\Y^{i}_{n+1}$ for $i=1,2$.
\end{proof}
\section*{Acknowledgment}
We thank the referee(s) for helpful suggestions on an earlier version of this paper.\\
Arij Benkhadra's research reported in this publication was supported by a scholarship from the Graduate Research Assistantships in Developing Countries Program of the Commission for Developing Countries of the International Mathematical Union.

\end{document}